\documentclass[final,1p,times]{elsarticle}

\usepackage{amssymb}
\usepackage{amsthm}
\usepackage{amsmath,amssymb,amsopn,amsfonts,mathrsfs,amsbsy,amscd}
\usepackage{varioref}
\usepackage{longtable}
\usepackage{multirow}

\newcommand{\al}{\alpha}
\newcommand{\esp}{\quad\mbox{and}\quad}
\newcommand{\lr}{\longrightarrow}
\newcommand{\R}{\mathbb{R}} 
\newcommand{\K}{\mathbb{K}} 

\newcommand{\D}{\mathrm{D}}
\newcommand{\om}{\omega}

\newcommand{\J}{{\mathfrak{J}}}

\newcommand{\h}{{\mathfrak{h}}}
\newcommand{\we}{\wedge}

\newtheorem{Def}{Definition}[section]
\newtheorem{theo}{Theorem}[section]
\newtheorem{pr}{Proposition}[section]
\newtheorem{Le}{Lemma}[section]
\newtheorem{co}{Corollary}[section]
\newtheorem{exem}{Example}
\newtheorem{remark}{Remark}

\labelformat{theorem}{théorème~#1}
\labelformat{lemma}{lemme~#1}
\labelformat{proposition}{proposition~#1}
\begin{document}
	
	\begin{frontmatter}

		\title{Cosymplectic Jacobi-Jordan Algebras}
		
		\author{S. El bourkadi and  M. W. Mansouri }
		\address{Department of Mathematics, Faculty of Sciences, Ibn Tofail University\\
			Analysis, Geometry and Applications Laboratory (LAGA)\\
			Kenitra, Morocco\\
			mansourimohammed.wadia@uit.ac.ma\\
			said.elbourkadi@uit.ac.ma}
		
		
		
		
		\begin{abstract}  
			We introduce the notion of cosymplectic structure on Jacobi-Jordan algebras, and we state that they are related to symplectic Jacobi-Jordan algebras. We show in particular, that they support a right-skew-symmetric product. We also study the double extension constructions of cosymplectic Jacobi-Jordan algebras and give a complete classification in dimension five.
		\end{abstract}
		
		
	\end{frontmatter}
	
	\section{Introduction}
	
	A Jacobi-Jordan algebra, denoted by JJ-algebra for short, is a pair $(\J,\cdot)$ consisting of a vector space $\J$ and a product $``\cdot"$, satisfying the following identities
	\begin{enumerate} 
		\item[(i)] $x\cdot y=y\cdot x$ (commutativity).
		\item[(ii)] $\oint x\cdot (y\cdot z):=x\cdot (y\cdot z)+y\cdot (z\cdot x)+z\cdot (x\cdot y) = 0$ (Jacobi identity),
	\end{enumerate}
	for any $x,y,z\in \J$, where $\oint$ denotes summation over the cyclic permutation.
	In other words, JJ-algebras are commutative algebras which their  products satisfy the Jacobi identity \cite{B-F}. JJ-algebras are introduced by Burde and Fialowski in \cite{B-F}. Note that this type of algebra has been known and appears under different names in the literature (e.g., Jordan algebras of nil index $3$ in \cite{W}, Lie-Jordan algebras in \cite{K-O}, and mock-Lie algebras in \cite{Z}).

	An important class of non-associative algebras is given by \emph{Jordan algebras}. Which are commutative non-associative algebras satisfying the identity: $x^2 \cdot(y\cdot x) = (x^2\cdot y)\cdot x$, it follows that JJ-algebra is a subclass of Jordan algebra, i.e., Any JJ-algebra is a Jordan algebra (see \cite{B} and \cite{B-F}).
	
	JJ-algebra and related topics have attracted significant attention in recent years. This interest stems from both the pure geometrical and algebraic viewpoint. In particular, several authors introduce and study additional structures of geometric origin ( symplectic structure) on JJ-algebras, see \cite{B-B} and \cite{A}. 
	Libermann in \cite{L} defined cosymplectic manifolds as $(2n+1)$-dimensional manifolds equipped with a closed $1$-form $\al$ and closed $2$-form $\om$ such that $\al\wedge\om^n$ is a volume form. 
	In light of the previous definition and the frameworks of cosymplectic geometry  (see \cite{L-S}, \cite{E-M} and \cite{F-V}), and the connection between symplectic and cosymplectic structures. All these elements motivate our investigation of cosymplectic JJ-algebras.
	
	Our principal objective is to introduce the notion of cosymplectic structure on JJ-algebras, explore their basic properties, and characterize the relationship between symplectic JJ-algebras and cosymplectic JJ-algebras.
	
	This paper goes as follows: Section $2$ lays the background and preliminary concepts and some constructions of JJ-algebras. This includes a review and a discussion of well-known results.  Additionally, we present established results concerning symplectic JJ-algebras, which will serve as a foundation for the subsequent analysis. In section $3$, we define cosymplectic structure on JJ-algebra. This structure is a  pair $(\al,\om)$, where  $\al$ is a $1$-form and $\om$ is a $2$-form, satisfying certain conditions.  
	Furthermore, we study the relation between symplectic and cosymplectic JJ-algebras. In particular, we show that there exists a one-to-one correspondence between cosymplectic and symplectic JJ-algebras equipped with a two-step nilpotent symplectic anti-derivation. Section $4$, we adapt the techniques of double extensions as in symplectic JJ-algebras presented in \cite{B-B} to describe cosymplectic JJ-algebra,  and we prove  that any cosymplectic JJ-algebra is obtained as a double extension of a cosymplectic JJ -algebra. Finally, in Section $5$, using the results of section $3$, we give a complete classification, up to isomorphism, of five dimensional cosymplectic JJ-algebra.
	
	\textbf{Notation.}  Let $\{e_i\}_{1\leq i\leq n}$ stands for a basis of $\J$,
	and let $\{e^i\}_{1\leq i\leq n}$ be its
	dual basis on $\J^\ast$ and  $e^{ij}$  the 2-form $e^i\wedge
	e^j\in\wedge^2\J^*$. Set by  $\langle x \rangle:= \mathrm{span}_\K\{x \}$ the one-dimensional JJ-algebra spanned by $x$.
	
	\section{Preliminaries on Jacobi-Jordan algebras}
	In this section, we recall some basic definitions and properties of JJ-algebra and symplectic JJ-algebra. 
	
	Here and subsequently, all algebras are considered of finite dimension over a field $\K$ of characteristic different from $2$ or $3$. Let $(\J,\odot)$ be an algebra.
	\begin{Def} The algebra $(\J,\odot)$ is said to be an anti-associative algebra if the product $\odot$ satisfies the identity 		
		$$(x\odot y)\odot z=-x\odot (y\odot z), $$		
		for any $x,y,z\in \J$.
	\end{Def}	
	
	We will use the symbol $(x,y,z)^{\odot}$ to denote the anti-associator of three elements $x,y,z\in\J$, it is defined by	$(x,y,z)^{\odot}:=x\odot (y\odot z)+(x\odot y)\odot z$.
	
	The algebra $(\J,\odot)$ is said to be an anti-associative algebra if $(x,y,z)^{\odot}=0$, for any $x,y,z\in\J$.
	
	\begin{Def}
		Let $(\J,\odot)$ be an algebra, Consider a new product on $\J$ denoted by $``\cdot"$ and defined by 
		\[x\cdot y := x\odot y + y\odot x,\]
		for any $x,y\in \J$. Then, $(\J,\odot)$ is called an admissible JJ-algebra if $(\J,\cdot)$ is a JJ-algebra. Moreover, the product $\odot$ is called an admissible JJ-product.
	\end{Def}
	As an example, a direct calculation shows that any anti-associative algebra is an admissible JJ-algebra.
	
	\begin{Le}
		An algebra $(\J,\odot)$ is an admissible JJ-algebra if and only if \[\oint (x,y,z)^{\odot}=-\oint (x,z,y)^{\odot},\]
		for any $x,y,z\in \J$.
	\end{Le}
	\begin{proof}
		It is a straightforward calculation.
	\end{proof}
	\begin{Def}
		Let $(\J,\odot)$ be an algebra. Then
		\item[$1.$] $(\J,\odot)$ is called right-skew-symmetric algebra if 
		\[(x,y,z)^{\odot}+(x,z,y)^{\odot}=0, \]
		\item[$2.$] $(\J,\odot)$ is called left-skew-symmetric algebra if 
		\[(x,y,z)^{\odot}+(y,x,z)^{\odot}=0,\]
		for all $x,y,z\in \J$.
	\end{Def}
	
	\begin{co}
		Any right(left)-skew-symmetric algebra is an admissible JJ-algebra. \label{right-adm}
	\end{co}
	\begin{proof}
		It follows from the previous lemma. 
	\end{proof}
	\begin{Def}
		Let $(\J,\cdot)$ be an algebra, and $\D$ be a linear map $\D: \J \lr \J $.
		\begin{enumerate} 
			\item[$(1)$] $\D$ is called a derivation of $(\J,\cdot)$ if  \[\D(x\cdot y)= \D(x)\cdot y+x\cdot \D(y), \]
			\item[$(2)$] $\D$ is called an anti-derivation of $(\J,\cdot)$ if  \[\D(x\cdot y)= -\D(x)\cdot y-x\cdot \D(y),\]
			for any $x,y\in \J$.		
		\end{enumerate}
		
	\end{Def}	
	Let $\mathrm{Der}(\J)$ denotes the set of all derivations of the algebra $(\J,\cdot)$ and $\mathrm{Ader}(\J)$ stands for the set of all anti-derivations of the algebra $(\J,\cdot)$.
	\begin{remark} Let $(\J,\cdot)$ be a JJ-algebra. For all $x\in \J$, we denote by $\mathrm{L}_x$ and $\mathrm{R}_x$ the left and right-multiplication on $(\J,\cdot)$, respectively. Since $(\J,\cdot)$ is commutative, it follows trivially that $\mathrm{L}_x=\mathrm{R}_x$. It follows from the Jacobi identity that $\mathrm{L}_x$ and $\mathrm{R}_x$ are anti-derivations of $(\J,\cdot)$.
	\end{remark}
	
		\textbf{JJ-algebras extensions}. Motivated by the classic notion of double extensions of Lie algebras, we adapt the same approach here with JJ-algebras as follows.

	Let $(\J,\cdot)$ be a JJ-algebra and $\theta$ a symmetric bilinear form on $\J$.
	
	On the vector space $\J_{(\theta,e)}=\J\oplus \langle e\rangle$, we define the product $``\cdot_\theta"$ by \begin{align*}
		x\cdot_\theta y &= x\cdot y+\theta(x,y)e,
	\end{align*}
	for every $x,y\in\J$. It is clear that $(\J_{(\theta,e)},\cdot_\theta)$  is a JJ-algebra if and only if $\theta $ satisfies \begin{equation}\oint\theta(x\cdot y,z)=0,\label{cocycle}
	\end{equation}
	for any $x,y,z\in \J$.
We call  $(\J_{(\theta,e)},\cdot_\theta)$ the central extension of $\J$ by $\theta$.

Let $\tilde{\J}=\langle d\rangle\oplus \J_{(\theta,e)}$ be the direct sum of $ \J_{(\theta,e)}$ with the one-dimensional vector space $\langle d\rangle$  and $\D$ be a linear map $\D: \J_{(\theta,e)} \lr \J_{(\theta,e)} $ . Consider the following product on $\tilde{\J}$ denoted by $``\bullet"$ and defined by
\begin{align}  \label{nrod}
	x\bullet y &= x\cdot y+\theta(x,y)e,\nonumber \\
	d \bullet x&= \D(x), \\
	d\bullet d &= a\in\J,\nonumber 
\end{align}
Assume that the product  $``\bullet"$ satisfies the Jacobi identity. Then, a direct calculation ensures the following proposition.
\begin{pr}\label{pr 2.1}
	Let $(\J,\cdot)$ be a JJ-algebra, then the vector space $\tilde{\J}=\langle d\rangle\oplus\J \oplus \langle e\rangle\ $ equipped with the product $``\bullet"$ defined by \eqref{nrod}. Then $(\tilde{\J},\bullet)$ is a JJ-algebra if and only if 
\begin{equation}\label{cp}
	\D\in\mathrm{Ader}(\J), \quad a\in\ker(\D)\esp \D^2=-\frac12\mathrm{L}_a. 
\end{equation}

\end{pr}
\begin{remark}
\begin{enumerate}
\item	If a pair $(\D,a)$ satisfies \eqref{cp}, then it is called an
	 admissible pair of $(\J,\cdot)$ l (cf. \cite{B-B}, Definition $2.5$) .	 
\item 	Consider $\D(x)= \varphi(x)+\lambda(x)e$, for any $x\in \J$ where $(\varphi,\lambda)\in \mathrm{End}(\J)\times \J^*$. Then  the condition \eqref{cp} is equivalent to $(\varphi,a)$ is an admissible pair of $(\J,\cdot)$ and 
	\begin{equation*}\label{c0}
	\left\{
	\begin{array}{rl}
		 \lambda(a)&=0,\\
	\lambda(x\cdot y)&=-\theta(\varphi(x),y)-\theta(\varphi(y),x),\\
	 \theta(a,x)&=-2\lambda\circ \varphi(x),
	\end{array}
	\right.
\end{equation*}
for any $x, y \in \J$. 
\end{enumerate}
\end{remark}	
	We say that $(\tilde{\J},\bullet)$ is a \emph{double extension} of JJ-algebra $(\J,\cdot)$ by $(\theta,\D,a)$ or  $(\theta,\varphi,\lambda,a)$.

		\textbf{Symplectic JJ-algebras}.	Let $(\h,\cdot)$ be a JJ-algebra. We recall (cf. \cite{B-B}, Definition $3.1$) that a symplectic structure on $\h$ is given by a skew-symmetric non-degenerate bilinear form $\om$ satisfying the following condition
	\begin{equation*}
		\oint\om(x,y,z):=\om(x\cdot y,z)+\om(z\cdot x,y)+\om(y\cdot z,x)=0, 
	\end{equation*}
	for any $x, y, z\in\h$.		
	If a JJ-algebra $\h$ is equipped with such symplectic structure $\om$, then $(\h,\om)$ is called a symplectic JJ-algebra.
	\begin{remark}
		It makes sense to reconsider the condition of non-degeneracy of $\om$. In other words, saying that $\om$ is non-degenerate is equivalent to stating that $\om^n\neq 0$, where 
		$\om^n:=\om \wedge\underbrace{\cdots}_\text{$n$-times}\wedge \om$.
	\end{remark}
	Furthermore (see also \cite{B-B}, Proposition $3.1$), there exists a product on the vector space $\h$ denoted  $``\odot"$ and given by 
	\begin{equation}\label{symproduct}
		\om(x\odot y,z)=\om(x,y\cdot z),
	\end{equation}
	and satisfies
	\begin{enumerate} 
		\item[$(1)$]$x\cdot y=x\odot y+y\odot x$\;\;(admissible) 
		\item[$(2)$] $(x,y,z)^{\odot}+(x,z,y)^{\odot}=0$,\;\;(right-skew-symmetric)
	\end{enumerate}
	for any $x, y, z\in\h$.
	
	The product $``\odot"$ is called the right-skew-symmetric product associated with the symplectic JJ-algebra $(\h,\om)$. 		
	
		Recall  that two symplectic JJ-algebras $(\h_1,\om_1)$ and $(\h_2,\om_2)$ are isomorphic if there exists an isomorphism of JJ-algebras  $\phi :\h_1\lr \h_2$ such that 
	\[ \om_2(\phi(x),\phi(y))=\om_1(x,y),\]
	for any $x, y\in\h_1$.
	\begin{exem}
		Let $(H_4, \cdot)$ be the four-dimensional JJ-algebra defined
		above by		
		\[ e_1 \cdot e_1 = e_2,\; e_1 \cdot e_3 = e_3 \cdot e_1 = e_4,\]
		where $\{e_1, e_2, e_3, e_4\}$ is a basis of $H_4$. The skew-symmetric bilinear form
		$\om=e^{12}+2e^{23}$ is a symplectic form on $H_4$. Moreover,	any four-dimensional symplectic JJ-algebra is isomorphic to the unique symplectic JJ-algebra $(H_4,\om)$ (see, \cite{B-B}).
	\end{exem}

	\section{Cosymplectic Jacobi-Jordan algebras}
	In this section we give a definition of cosymplectic structure on JJ-algebras. In particular, we state that the dimension of cosymplectic JJ-algebra is always necessarily odd.
	We establish that every cosymplectic JJ-algebra induces a corresponding symplectic JJ-algebra. Furthermore, under certain conditions and with additional data, we show that there is a one-to-one correspondence between cosymplectic and symplectic JJ-algebras.   
	\begin{Def} An almost cosymplectic structure on a JJ-algebra $(\J,\cdot)$ is a pair  $(\al,\om)$, where $\al$ is a $1$-form  and $\om$ is a skew-symmetric bilinear form such that $\al \wedge \om^n\neq 0$, the triplet $(\J,\al,\om)$ is called  almost cosymplectic JJ-algebra.
		
		The almost cosymplectic JJ-algebra $(\J,\al,\om)$  is said to be a  cosymplectic JJ-algebra if $\al$ and $\om$ satisfy the following conditions 
		\begin{equation}\label{alclosed}
			\al(x\cdot y)	=0,
		\end{equation}
		\begin{equation}\label{omclosed}		 
			\oint\om(x\cdot y,z)=0, 
		\end{equation}
		for any $x, y, z\in\J$.
	\end{Def}
	From the fact that $ \al \wedge \om^n\neq 0$, any  cosymplectic structure $(\al,\om)$ on $(\J,\cdot)$ induces an isomorphism $\Psi$ given by \begin{equation}\label{1}
		\begin{array}{rcl}
			\Psi : \J& \lr & \J^* \\
			x & \longmapsto & \om(x,.) +\al(x)\al(.),
		\end{array}
	\end{equation}
	for any $x\in\J$. Consequently, the unique vector $\xi=\Psi^{-1}(\al)$ on $\J$ is called \textit{Reeb vector} of the cosymplectic JJ-algebra  $(\J,\al,\om)$, and it is completely described by the following conditions
	\begin{align}
		\al(\xi)&=1,\\
		\om(\xi,.)&=0.\label{omxi}
	\end{align}
	\begin{remark}
		\begin{enumerate}
			\item Substituting $z=\xi$ into $\eqref{omclosed}$, we find that
			\begin{equation}\label{der}
				\om(x,y\cdot\xi)+\om(y,x\cdot\xi)=0, 
			\end{equation}
			for any $x, y\in\J$.	
			\item In particular, if $y=\xi$, the equation above gives that $\om(x,\xi \cdot\xi)=0$, for all $x\in \J $. Therefore, $\xi^{2}=\xi\cdot \xi=0\label{xinul}$. 
			
			\item As in the case with symplectic JJ-algebras being even-dimensional, the condition $\al \we \om^{n}\neq 0$, ensures that any cosymplectic JJ-algebra must have odd dimension.
		\end{enumerate}
	\end{remark}
	\begin{Le}\label{Le3.1}
		Let $(\J,\al,\om)$ be a cosymplectic JJ-algebra.
		Then, $\h=\mathrm{ker(\al)}$ is an ideal of $\J$ and $ (\h,\om_{\h}) $ is a symplectic JJ-algebra.
	\end{Le}
	
	\begin{proof} By \eqref{alclosed}, we have $\J\cdot \J\subset \h$. In particular, $\h$ is an ideal of $\J$. Set by $\om_{\h}:=\om_{|\h\times\h }$  the restriction of $\om$ on $\h$. Let us prove that  $ (\h,\om_{\h})$ is a symplectic JJ-algebra. It is clear that $\om_\h$ is a skew-symmetric  bilinear form on $\h$ and satisfies \eqref{omclosed}. 
		Finally, let $\{\xi,e_1,...,e_{2n}\}$ be a basis of $\J$, and $\{e_1,...,e_{2n}\}$ be a basis of $\h$. We have
		\[0\not=\al\we\om^n(\xi,e_1,...,e_{2n})=\om_{\h}^n(e_1,...,e_{2n}),\]
		i.e., $\om_{\h}$ is  non-degenerate. Therefore,  $(\h,\om_\h)$ is a symplectic JJ-algebra.
	\end{proof} 
	Let $(\J,\al,\om)$ be a cosymplectic JJ-algebra with Reeb vector $\xi$, and the product $\odot$ be the right-skew-symmetric product associated with the symplectic JJ-algebra $(\h,\om_{\h})$ as given in \eqref{symproduct}.
	
	Set $\D_{\xi}$ to be a linear map $\D_{\xi}: \J \lr \J $ such that $\D_{\xi}(x):=\xi\cdot x$, for any $x\in \J$. It follows from \eqref{der} that $\D_{\xi}$ is an anti-derivation of $(\J,\cdot)$. 
	\begin{Le} 	
		The linear map $\mathrm{D}_{\xi}$, as defined above, is an anti-derivation of the algebra $(\h,\odot)$, i.e.,
		\[ \mathrm{D}_{\xi}(x \odot y)=-\mathrm{D}_{\xi}(x) \odot y-x \odot \mathrm{D}_{\xi}(y), \]
		for any $x,y\in h$.
	\end{Le}
	\begin{proof}
		Let $x,y,z\in \h$, we have
		\begin{align*}
			\om_{\h}(\mathrm{D}_{\xi}(x\odot y)+\mathrm{D}_{\xi}(x)\odot y,z)&=\om_{\h}(\xi \cdot(x\odot y),z)+\om_{\h}((x\cdot\xi)\odot y,z)\\
			&=\om_{\h}(x\odot y,\xi \cdot z)+\om_{\h}(x\cdot\xi, y\cdot z)\\
			&=\om_{\h}(x,y\cdot(\xi \cdot z))+\om_{\h}(x, \xi\cdot (y\cdot z))\\
			&=\om_{\h}(x,-z\cdot(y\cdot \xi))\\
			&=\om_{\h}(-x\odot(y\cdot \xi),z)\\
			&=\om_{\h}(-x\odot \mathrm{D}_{\xi}(y),z).
		\end{align*}
		The non-degeneracy of $\om_{\h}$ ensures the result.		
	\end{proof}
	The following proposition shows that any cosymplectic JJ-algebra give rise to a right-skew-symmetric algebra.
	\begin{pr}
		Let $(\J,\al,\om)$ be a cosymplectic JJ-algebra. There exists a product $``*"$ on the underlying vector space $\J$  defined by
		\begin{equation} \label{isom}
			\Psi(x*y)(z)= \Psi(x)(y\cdot z),
		\end{equation}
		for any $x, y, z\in\J$, and it satisfies
		\begin{enumerate}
			\item  For all $x,\,y\in\h$
			\begin{equation}\label{Pr1}
				x*y=x\odot y+\om(x,y\cdot\xi)\xi.
			\end{equation}  
			
			\item For all $ x\in\J$
			\begin{equation}\label{Pr2}
				x*\xi=x \cdot \xi \esp  \xi* x =0.
			\end{equation}
			
			\item The algebra $(\J,*)$ is a right-skew-symmetric algebra. 
		\end{enumerate}		
	\end{pr}
	\begin{proof}
		The fact that $\Psi$ is an isomorphism induces a well-defined product on $\J$ denoted by $``*"$ and given by \eqref{isom} , in other terms for all $x,y\in\J$, and $z\in\J$, the equation \eqref{isom} is equivalent to
		\begin{equation}
			\om(x*y,z) +\al(x*y)\al(z)=\om(x,y\cdot z) +\al(x)\al(y\cdot z).\label{prod}
		\end{equation}
		\begin{enumerate}
			\item Consequently, for all $x,y\in\h$, and $z\in\J$, the equation \eqref{prod} becomes
			\begin{align*}
				\om(x*y,z)+\al(x*y)\al(z)&=\om(x,y\cdot z) 
			\end{align*}
			Substituting $z=\xi$ in the equation above, it follows that 			
			\begin{align*}
				\al(x*y)&=\om(x,y\cdot \xi).
			\end{align*}
			In particular, for all $z\in\h$,  we have that $\om_{\h}(x*y,z)=\om_{\h}(x,y\cdot z)$, which means that $``*"$ reduces to $``\odot"$ on the vector space $\h$, and  This shows 1.
			\item If $x=\xi$ and  $y\in\J$, the relation $\eqref{prod}$ becomes
			\begin{align*}
				\om(\xi*y,z) +\al(\xi*y)\al(z)=0.
			\end{align*}
			On the one hand, replacing $z$ by $\xi$ we obtain $\al(\xi* y)=0,$ this gives  $\xi*y\in\h$. On the other hand for all $z\in\h$, we have 
			\begin{align*}
				\om_{\h}(\xi*y,z)&=0,
			\end{align*}
			the non-degeneracy of $\om_{\h}$ shows that $\xi*y = 0$ for all $y\in\J$.
			
			For all $x\in\J$, if $y=\xi$, we also have
			\begin{align*}
				\om(x*\xi,z)+\al(x*\xi)\al(z)&=\om(x,\xi \cdot z). 
			\end{align*} 
			Let $z=\xi$, we have $\al(x*\xi)=0,$ this implies that $x*\xi\in\h$. Also for all $z\in\h$
			\begin{align*}
				\om_{\h}(x*\xi,z)&=\om_{\h}(x\cdot\xi,z). 
			\end{align*}
			As $\om_{\h}$ is non-degenerate, it follows that $x*\xi=x\cdot\xi,$ for all $x\in\J.$  
			\item On the one hand, for all $x,y$ and $z\in\J$, by \eqref{Pr1} and \eqref{Pr2}  we have	
			\begin{align*}
				(x,y,z)^{*}&=\big(x\odot y+\om(x,y\cdot\xi)\xi\big)*z+x*\big(y\odot z+\om(y,z\cdot\xi)\xi\big)\\
				&=(x\odot y)\odot z+\om(x\odot y,z\cdot\xi)\xi+x\odot (y\odot z)+\om(x,(y\odot z)\cdot\xi)\xi+\om(y,z\cdot\xi)x\cdot\xi\\
				&=(x\odot y)\odot z+x\odot (y\odot z)+ \om(x\odot y,z\cdot\xi)\xi+\om(x,(y\odot z)\cdot\xi)\xi+\om(y,z\cdot\xi)x\cdot\xi\\
				&=(x,y,z)^{\odot}+ \om(x\odot y,z\cdot\xi)\xi+\om(x,(y\odot z)\cdot\xi)\xi+\om(y,z\cdot\xi)x\cdot\xi.
			\end{align*}
			Thus,
			\begin{align*}
				(x,y,z)^{*}+(x,z,y)^{*}&=(x,y,z)^{\odot}+(x,z,y)^{\odot}+\om(x\odot y,z\cdot\xi)\xi+\om(x,(y\odot z)\cdot\xi)\xi\\
				&\;+\om(x\odot z,y\cdot\xi)\xi+\om(x,(z\odot y)\cdot\xi)\xi +\big(\om(y,z\cdot\xi)+\om(z,y\cdot \xi)\big)x\cdot\xi.
			\end{align*}	 	 
			Since $(h,\odot)$ is right-skew-symmetric (JJ-admissible) algebra, it follows that
			\begin{align*}
				(x,y,z)^{*}+(x,z,y)^{*}&=\big(\om(x,(z\cdot \xi)\cdot y)+\om(x,(y\cdot z)\cdot \xi)+\om(x,(\xi\cdot y)\cdot z\big)\xi-\om(\xi,y\cdot z)x\cdot\xi\\
				&=\om\big( x,(z\cdot \xi)\cdot y+(y\cdot z)\cdot \xi+(\xi\cdot y)\cdot z\big)\xi\\
				&=\om(x,\oint (z\cdot \xi)\cdot y)\\
				&=0.
			\end{align*}
			On the other hand, for all $x,y\in \h$.
			\begin{align*}
				(x,\xi,y)^{*}+(x,y,\xi)^{*}&=(x\cdot \xi)*y+\big(x\odot y+\om(x,(y\cdot \xi)\big)\cdot\xi+x\odot (y\cdot\xi)+\om\big(x,(x\cdot \xi)\big)\xi\\
				&=\D_{\xi}(x)\odot y+\om\big(x\cdot \xi,y\cdot \xi\big)\xi+\D_{\xi}(x\odot y) +\om \big(x,(y\cdot \xi)\cdot \xi\big)\xi +x\odot \D_{\xi}(y)\\
				&=\om \big(\xi,(y\cdot \xi)\cdot x\big)\xi\\
				&=0.
			\end{align*}
		\end{enumerate}
		Finally, we conclude that $(\J,*)$ is a right-skew-symmetric JJ-algebra.
	\end{proof}
	\begin{remark}
		Since $(\J,*)$ is a right-skew-symmetric JJ-algebra, it follows from Corollary \ref{right-adm} that $(\J,*)$ is an admissible JJ-algebra.
	\end{remark}
	
	\begin{theo}\label{theo1}
		Any	$(2n+1)$-dimensional cosymplectic JJ-algebra $(\J,\al,\om)$ is in one-to-one correspondence with a symplectic JJ-algebra $(\h,\om_\h,\D)$  of dimension $2n$  equipped with an anti-derivation $\D$ satisfying 
		\begin{equation}\label{omD}
			D^2(x)=0 \esp \om_\h(\D(x),y)=\om_\h(x,\D(y)),
		\end{equation}
		for any $x,y\in\h$.
	\end{theo}	
	\begin{proof}
		Let $(\J,\al,\om)$ be a $(2n+1)$-dimensional cosymplectic JJ-algebra and $``\bullet"$ its product. Set $\h=\mathrm{ker}(\al)$ and $\om_{\h}:=\om_{|\h\times\h }$, it follows from Lemma \ref{Le3.1} that $(\h,\om_{\h})$ is a symplectic JJ-algebra.
		
		For all $x\in \h$, we have $ \xi\bullet x =\D_{\xi}(x) \in \h$, one can take $\D=\D_{\xi}\in \mathrm{Ader}(\h)$. By \eqref{der} we have  $\om_\h(\D_{\xi}(x),y)=\om_\h(x,\D_{\xi}(y))$, for any $x,y\in\h$.
		
Conversely, let $(\h,\om_\h,\D)$ be a symplectic JJ-algebra of dimension $2n$ endowed with an anti-derivation $\mathrm{D}$ satisfying the conditions above and $``\cdot"$ its product. It is clear to see that  $(\J=\langle \xi \rangle \oplus \h,\bullet) $ is a well defined JJ-algebra, such that
\begin{align*}
x\bullet y&=x\cdot y,\\
	 \xi \bullet x &=\D(x),\\
	\xi\bullet \xi&=0,
	\end{align*}

for any $x,y\in \h$. We extend $\om_{\h} $ on $\J$ by defining  $\om(\xi,.)=0$ and $\om(x,y)=\om_{\h}(x,y)$  for any $x,y\in \h$. Since $\om_{\h}$ is symplectic together with \eqref{omD}, this implies that $\om$ satisfies \eqref{omclosed}. 
		
Let $\al$ be a $1$-form on $\J$ defined by $\al_{\vert\h}=0$ and  $\al(\xi)=1$. The fact that $\mathrm{D}(x)=x\bullet\xi\in \h$, implies that $\al(\xi \bullet x )=0$. Hence $\al(x \bullet y)=0$, for any $x,y\in \J$.
		
		Finally, let $\{\xi,e_1,...,e_{2n}\}$ be a basis of $\J$, and $\{e_1,...,e_{2n}\}$ be a basis of $\h$. We have
		\[\al\we\om^n(\xi,e_1,...,e_{2n})=\om_{\h}^n(e_1,...,e_{2n})\not=0,\]
		Thus, the skew-symmetric bilinear form $\om_{\h}$ is non-degenerate. We deduce that  $(\J,\al,\om)$ is a $(2n+1)$-dimensional cosymplectic JJ-algebra.	
	\end{proof}	
	\begin{Def}	
		Two cosymplectic JJ-algebras $(\J_1,\al_1,\omega_1)$ and $(\J_2,\al_2,\omega_2)$ are isomorphic if there exists an isomorphism of JJ-algebras  $\Phi : \J_1\lr \J_2$ such that 
		\[ \omega_2(\Phi(x),\Phi(y))=\omega_1(x,y)\esp \al_2\circ \Phi(x)=\al_1(x),  \]
		for any $x,y\in \J_1$.
	\end{Def}
	
	\begin{remark}
		In addition to the conditions mentioned above, for two cosymplectic JJ-algebras $(\J_1,\al_1,\om_1)$ and $(\J_2,\al_2,\om_2)$ to be isomorphic, we present an extra condition, which is  $\al_1(\xi_1)=\xi_2$. Indeed, let $\xi_{i}$ be the Reeb vector of $\J_{i}$, for $i=1,2$. Note that
		$\al_2\big(\phi(\xi_1)\big)=\al_1(\xi_1)=1$ and $\om_2(\phi(\xi_1),.)=\om_1(\xi_1,.)=0$, Therefore $\phi(\xi_1)=\xi_2$.
	\end{remark}
	
	\begin{pr}	\label{Pr3}
		Let  $(\J_1,\al_1,\om_1)$ and  $(\J_2,\al_2,\om_2)$ two  cosymplectic JJ-algebras, with  $(\h_1,\om_{\h_1},\D_1)$ and  $(\h_2,\om_{\h_2},\D_2)$
		the corresponding symplectic JJ-algebras and anti-derivations as in Theorem~\ref{theo1}.  Then, $(\J_1,\al_1,\om_1)$ and  $(\J_2,\al_2,\om_2)$ are isomorphic if and	only if there exists an isomorphism   $\phi : \h_1\lr \h_2$ of  symplectic JJ-algebras, such that   $\phi\circ\D_1 = \D_2\circ\phi$.
	\end{pr}	
	
	\begin{proof}
		Let $(\J_1,\al_1,\om_1)$ and $(\J_2,\al_2,\om_2)$ be two $(2n+1)$-dimensional cosymplectic JJ-algebras, assume that they are isomorphic by $\phi :\J_1 \longmapsto \J_2$, It follows from Lemma ~\ref{Le3.1} that $(\h_1,\om_{\h_1})$ and $(\h_2,\om_{\h_2})$ are symplectic JJ-algebras of dimension $2n$, where $\h_1=\mathrm{ker}(\al_1)$ and $\h_2=\mathrm{ker}(\al_2)$. In particular, $\phi :\h_1 \longmapsto \phi(\h_1)$ is an isomorphism of JJ-algebras and $\om_{\h_2}(\phi(\h_1),\phi(\h_1))=\om_{\h_1}(\h_1,\h_1)$ .
		
		For all $x\in \h_1$, we have $\al_1(x)=\al_2\big( \phi(x)\big)=0$, then $\phi(\h_1)\subset \h_2 $, since $\mathrm{dim}(\phi(\h_1))=\mathrm{dim}(\h_2)$, this implies that $\phi(\h_1)=\h_2$, we conclude that  $\phi :\h_1 \longmapsto \h_2$ is an isomorphism of symplectic JJ-algebras. By Theorem ~\ref{theo1}, we state that $\D_i(x)=\xi_i\cdot x$, for $i=1,2$ (without loss of generality, we fix $``\cdot"$ as product for all JJ-algebras). Then, for any $x\in \h_1 $, we have   $\phi\circ \D_1(x)=\phi(\xi_1\cdot x)=\xi_2\cdot \phi(x)=\D_2 \circ \phi(x)$.
		
		On the other hand, let $(\h_1,\om_{\h_1},\D_1)$ and  $(\h_2,\om_{\h_2},\D_2)$ be two isomorphic symplectic JJ-algebras with $\phi : \h_1\lr \h_2$ such that $\phi\circ\D_1 = \D_2\circ\phi$, where $\D_1$ and $\D_2$ two anti-derivations as in Theorem ~\ref{theo1}. Similarly to the second part of the proof of Theorem ~\ref{theo1}, we consider the vector spaces $(\J_i=\h_i \oplus \langle \xi_i \rangle,,\al_i,\Omega_i)$, for $i=1,2$. Then, we can extend $\phi$ to an isomorphism of JJ-algebras $\Phi :\J_1 \longmapsto \J_2$ by $\Phi|_{\h_1}=\phi$, and $\Phi(\xi_1)=\xi_2$. Indeed, $\Phi(\xi_1 \cdot \h_1)=\Phi\circ\D_1(\h_1)=\phi\circ\D_1(\h_1)=\D_2\circ\phi(\h_1)= \xi_2\cdot \h_2=\Phi(\xi_1)\cdot \Phi(\h_1)$. Its reminds to show that $\Phi$ is a cosymplectic isomorphism of JJ-algebras. On the one hand, it is clear that $\al_2\circ\Phi(\h_1)=\al_2(\h_2)=0=\al_1(\h_1)$, and  $\al_2\circ\Phi(\xi_1)=\al_2(\xi_2)=1  =\al_1(\xi_1)$. On the other hand, it is straightforward to see that  $\omega_2\big(\Phi(\h_1),\Phi(\h_1)\big)=\om_{\h_2}\big(\h_2,\h_2 \big)= \om_{\h_2}\big(\phi(\h_1),\phi(\h_1)\big)=\omega_1(\h_1,\h_1)$. Finally, $\omega_2\big(\Phi(\xi_1),\Phi(\h_1)\big)=\om_{\h_2}\big(\xi_2,\h_2 \big)=0=\omega_1(\xi_1,\h_1)$, which completes the proof.
	\end{proof}	
	If Theorem ~\ref{theo1} establishes a connection between $(2n+1)$-dimensional cosymplectic JJ-algebras and $2n$-dimensional symplectic JJ-algebras, then the following proposition presents the opposite effect. It establishes a relationship between a $(2n+1)$-dimensional cosymplectic JJ-algebra and a $(2n+2)$-dimensional symplectic JJ-algebras under certain conditions.
	\begin{pr} \label{prop}
		Let $(\J,\cdot)$ be a JJ-algebra, and $\al$, $\om$ be a $1$-form and $2$-form on $\J$, respectively. On the direct sum JJ-algebra $(\J\oplus \langle e\rangle,\cdot)$, we consider the $2$-form $\Omega=\om+\alpha\wedge e^{*}$. Then, $(\J,\al,\om) $ is a cosymplectic JJ-algebra if and only if $(\J\oplus \langle e\rangle,\Omega)$ is a symplectic JJ-algebra.
	\end{pr}
	\begin{proof}
		Suppose that $(\J,\al,\om) $ is a cosymplectic JJ-algebra. In view of Theorem~\ref{theo1}, the vector space $\J$ always can be written as $\J=\h \oplus \langle \xi \rangle$, where $(\h,\om_{\h})$ is a symplectic JJ-algebra.
		
		Note that the vector space $\J \oplus \langle e\rangle$ is a direct sum of $\J$ with the one-dimensional vector space $\langle e\rangle$. Therefore, $(\J \oplus \langle e\rangle,\cdot)$  is a JJ-algebra.
		
		Let $\Omega$ be a skew-symmetric $2$-form on $\J \oplus \langle e\rangle$ defined as follows
		\[\Omega(x,y)=\om_{\h}(x,y), \;\;\;\;\;\Omega(\xi,e)=1 \esp \Omega(\xi,x)= \Omega(e,x)=0,\]
		for every $x,y\in\h$, it is clear that $\Omega$ is non-degenerate on $\J \oplus \langle e\rangle$. Hence, $(\J\oplus \langle e\rangle,\Omega)$ is a symplectic JJ-algebra.

		Conversely. Let $\J$ be a $(2n+1)$-dimensional JJ-algebra. Consider a $1$-form $\al$ and $2$-form $\om$ on $\J$. Assume that $(\J \oplus \langle e\rangle,\Omega)$ is a $(2n+2)$-dimensional symplectic JJ-algebra, where $\Omega=\om+\alpha\wedge e^{*}$.
		
		Since $\Omega$ is a symplectic structure, we have $\Omega^{n+1}\neq 0$. A straightforward calculation ensures that $\om^{n}\wedge \al \wedge e^{*}\neq 0$, which implies $\om^{n}\wedge \al \neq 0$ on $\J$.
		
		On the one hand, we have $\oint\Omega(x,y,z):=\om(x\cdot y,z)+\om(z\cdot x,y)+\om(y\cdot z,x)=0$ for all $x, y, z\in\J$, which makes $\om$ a skew-symmetric  bilinear form on $\J$ and satisfies the identity \eqref{omclosed}. 
		
		On the other hand, $\oint\Omega(x,y,e):=\al(x\cdot y)=0$, so $\al$ satisfies \eqref{alclosed}. We conclude that $(\J,\al,\om) $ is a cosymplectic JJ-algebra. 
	\end{proof}
	\section{Cosymplectic   JJ-algebras double extensions}

	Let $(\J,\al,\om)$ be a $(2n+1)$-dimensional  cosymplectic JJ-algebra and let	$(\tilde{\J}=\langle d\rangle\oplus\J \oplus \langle e\rangle,\bullet)$ be a double extension  of $(\J,\cdot)$  by $(\theta,\varphi,\lambda,a)$ and let $\tilde{\om}$ be a skew-symmetric $2$-form  and $\tilde{\al}$ a $1$-form  on $\tilde{\J}$ defined as follows
	\[\tilde{\om}=\om(x,y),\;\forall x,y\in\J,\; \tilde{\om}(d,e)=1,\;\tilde{\om}(d,.)=\tilde{\om}(e,.)=0,\]
	and 
	\[\tilde{\al}(x)=\al(x),\;\forall x\in\J,\; 
	\tilde{\al}(d)\in \K,\;\tilde{\al}(e)\in \K.\]
	On the one hand, suppose that $\tilde{\al}$ satisfies \eqref{alclosed}. It follows that
	\begin{equation}\label{c1}
		\left\{
		\begin{array}{l}
			\theta(x,y)\tilde{\al}(e)=0,\\
		\al\circ \varphi(x)+\lambda(x)\tilde{\al}(e)=0,\\
	\al(a)=0.
		\end{array}
		\right.
	\end{equation}
	On the other hand, suppose that $\tilde{\om}$ satisfies \eqref{omclosed}. Then, for any $x,y\in\J$, It follows that
	\begin{equation}\label{c2}
		\left\{
		\begin{array}{l}
	\theta(x,y)=\om_\varphi(x,y)\\
	2\lambda(x)=\om(a,x)\\
\om(\varphi(x),a)=0.
	\end{array}
		\right.
	\end{equation}

Where $\om_\varphi(x,y)=\om(\varphi(x),y)+\om(\varphi(y),x)$ for any $x$, $y\in\J$.	From the first row of the system \eqref{c1}, one can distinguish two cases.
	
	\textbf{The first case:} $\tilde{\al}(e)=0$. 
	
The conditions described above, together with Proposition \ref{pr 2.1}, yield the following result.

	\begin{theo}\label{the4.1}		Let $(\J,\al,\om)$ be a cosymplectic JJ-algebra and $(\varphi,a)$  an admissible pair of $(\J,\cdot)$. Then  $\tilde{\J}=\langle d\rangle\oplus\J \oplus \langle e\rangle\ $ equipped with the product 
		\begin{align*}  
			x\bullet y &= x\cdot y+\om_\varphi(x,y)e, \\
			d \bullet x&= \varphi(x)+\frac{1}{2}\om(a,x)e, \\
			d\bullet d &= a\in\J,
		\end{align*}
for any $x$ and $y\in\J$. Then,
		\[\tilde{\al}_{|\J}=\al+td^*\esp\tilde{\om}=\om + d^{*}\we e^{*},\quad t\in \R\]
is a  cosymplectic structure on  $\tilde{\J}$ if  $\al\circ \varphi=0$ and $\om(\varphi(x),a)=0$, for any $x\in\J$.		
\end{theo}
	\begin{proof}
		To complete the proof, it suffices to show that $\tilde{\al}\we \tilde{\om}^{(n+1)}\neq 0$. This is equivalent to proving that $\Psi : \tilde{\J} \lr \tilde{\J}^{*}$,  $\tilde{x}  \longmapsto  \tilde{\om}(\tilde{x},.) +\tilde{\al}(\tilde{x})\tilde{\al}(.)$ is an isomorphism. Let $\tilde{x}\in\langle d \rangle\oplus\J \oplus\langle e \rangle$ and write  $\tilde{x}=x_1d+x+x_2\xi+x_3e$, where $x\in\h=\mathrm{ker(\al)}$ and $x_1$, $x_2$, $x_3\in\R$. One can directly verify that
		\[\left\{
		\begin{array}{ll}
			\Psi(\tilde{x},y )&=\om_{\h}(x,y),\\
			\Psi(\tilde{x},e)&=x_1,\\
			\Psi(\tilde{x},\xi)&=x_1 \tilde{\al}(d)+x_2,\\
			\Psi(\tilde{x},d)&=-x_3+(x_1\tilde{\al}(d)+x_2)\tilde{\al}(d).
		\end{array}
		\right.\]
		for any $y \in\h$. Since $\om_{\h}$ is a non-degenerate bilinear form, it follows that if $\Psi(\tilde{x},.)=0$, then $\tilde{x}=0$. Therefore $\Psi$ is injective. 
	\end{proof}

	\textbf{The second case:} 
	$\tilde{\al}(e)\not=0$. 
	
	For simplicity, one can take $\tilde{\al}(e)=1$. So, by \eqref{c1} and  \eqref{c2} we obtain that
	\begin{equation*}
		\theta(x,y)=0,\quad \al(a)=0 \esp   \lambda(x)=-\al\circ \varphi(x)=\frac{1}{2}\om(a,x),
	\end{equation*}
	\begin{equation*}
		\om(\varphi(x),y)=\om(x,\varphi(y)),
	\end{equation*}
	for any $x,y\in\J$. Under the above assumptions and  Proposition \ref{pr 2.1}, one  conclude that
	\begin{theo}
		Let $(\J,\al,\om)$ be a cosymplectic JJ-algebra and $(\varphi,a)$  an admissible pair of $(\J,\cdot)$. Then  $\tilde{\J}=\langle d\rangle\oplus\J \oplus \langle e\rangle\ $ equipped with the product 
	\begin{align*}  
		x\bullet y &= x\cdot y, \\
		d \bullet x&= \varphi(x)+\frac{1}{2}\om(a,x)e, \\
		d\bullet d &= a\in\J,
	\end{align*}
for any $x$ and $y\in\J$.	Then,  
		\[\tilde{\al}_{|\J}=\al+e^*+td^* 
		,\;\;\tilde{\om}=\om + d^{*}\we e^{*},\quad t\in \K\]
		defines a cosymplectic structure on $\tilde{\J}$ if
		$\om_\varphi=0$ and  $\lambda=\frac{1}{2}\om(a,.)=-\al\circ \varphi$. 
	\end{theo}	
	\begin{proof}
		As in the proof of Theorem \ref{the4.1}, we only need to show that  $\Psi : \tilde{\J} \lr \tilde{\J}^{*}$ is an isomorphism. Let $\tilde{x}\in\langle d \rangle\oplus\J \oplus\langle e \rangle$ and write  $\tilde{x}=x_1d+x+x_2\xi+x_3e$, where $x\in\h=\mathrm{ker(\al)}$ and $x_1,\,x_2,\,x_3\in\K$. Note that
		\[\left\{
		\begin{array}{ll}
			\Psi(\tilde{x},y )&=\om_{\h}(x,y),\; \forall y \in\h,\\
			\Psi(\tilde{x},\xi)&=x_1\tilde{\al}(d)+x_2-x_3,\\
			\Psi(\tilde{x},e)&=x_1-(x_1\tilde{\al}(d)+x_2-x_3),\\
			\Psi(\tilde{x},d)&=-x_3+(x_1\tilde{\al}(d)+x_2-x_3)\tilde{\al}(d).
		\end{array}
		\right.\]
		It is clear that if $\Psi(\tilde{x},.)=0$, then $\tilde{x}=0$.
	\end{proof}

\begin{theo}	
		Let $(\tilde{\J},\tilde\al,\tilde\om)$ be a cosymplectic JJ-algebra such that $\tilde{\J}\not=\{0\}$. Then,  $(\tilde{\J},\tilde\al,\tilde\om)$ is a double extension of cosymplectic JJ-algebra $(\J,\al,\om)$.
	\end{theo}	
\begin{proof}
Let $(\tilde{\J},\tilde\al,\tilde\om)$ be a cosymplectic JJ-algebra and $``\bullet"$ be its product. It follows from Lemma \ref{Le3.1} that $\tilde{\J}=\langle \xi \rangle\oplus \tilde \h$ where 
$\tilde \h= \ker (\tilde\al)$ and $ (\tilde\h,\tilde\om_{\tilde\h})$ is a symplectic JJ-algebra. Consequently, Theorem 3.3 \cite{B-B} implies that $ (\tilde\h,\tilde\om_{\tilde\h}) $ is a symplectic double extension of a symplectic JJ-algebra $ (\h,\om)$. i.e., there exists an admissible pair  $(\varphi,a)$ of $(\h,\cdot)$, such that $\tilde\h=\langle d \rangle\oplus\h \oplus\langle e \rangle$
	\begin{align*}  
	x\bullet y &= x\cdot y+\om_\varphi(x,y)e, \\
	d \bullet x&= \varphi(x)+\frac{1}{2}\om(a,x)e, \\
	d\bullet d &= a,
\end{align*}
 $\tilde\om_{\tilde \h}=\om+d^*\we e^*$ and $\om(\varphi(x),a)=0$, for any $x,y\in\h$. Note that $\om=\tilde\om_{\h}$ and $\tilde\al\circ\varphi=0$.

In $\tilde\J=\xi\oplus\tilde\h$, for all $x\in\tilde\h$, we have  
\begin{align*}  
	\xi\bullet x &= f(x)d+D(x)+l(x)e, \quad\\
	d \bullet \xi&=kd+v+te,\quad v\in\h, \; k,t\in\K,
\end{align*}
where  $D$ is an endomorphism of $\h$, $f$ and $l$ are linear forms of $\h$. Since $\tilde\om$ satisfies \eqref{omclosed}, then $f\equiv0$, $k=0$, $t=0$, $l(x)=\om(v,x)$  and $\tilde\om_\h(D(x),y)=\tilde\om_\h (x,D(y))$ for any $x,y\in \h$.

We extend $ \varphi$    on  $\J=\xi\oplus\h $ by $\varphi(\xi)=v$.
 
 The product becomes
 \begin{align*}  
 	x\bullet y &= x\cdot y+\tilde\om_{\varphi}(x,y)e, \\
 	\xi\bullet x &=D(x)+\tilde\om_{\varphi}(\xi,x)e,\\
 	d \bullet x&= \varphi(x)+\frac{1}{2}\tilde\om(a,x)e, \\ 	
 	d \bullet \xi&=\varphi(\xi),\\
 	d\bullet d &= a.
 \end{align*} 
  Since the product $``\bullet"$ satisfies Jacobi identity, then $\D \in \mathrm{Ader}(\J)$ such that $\D^2=0$. By Theorem \ref{theo1}  it follows that $(\J,,\tilde\al_{\J},\tilde\om_\J)$ is a cosymplectic JJ-algebra and $(\varphi,a)$ is an admissible pair of $(\J,\cdot)$. In addition, $\oint\om(\xi \bullet d,a)=0$ implies that $\tilde\om_\J(\varphi(\xi),a)=0$. Finally,  $(\tilde{\J},\tilde\al,\tilde\om)$   is cosymplectic JJ-algebra double extension of $(\J,\tilde\al_{\J},\tilde\om_\J)$.
\end{proof}

		\textbf{Non-trivial examples in every odd dimension}
		
	Using the previous constructions, we give examples of non-trivial cosymplectic JJ-algebra in every odd dimension. 
	
	Let $(\K^{2n+1},\om,\al)$, with $\om=\sum_{i=1}^ne^i\we e^{i+n}$ and $\al=e^{2n+1}$  be the trivial $(2n+1)$-dimensional cosymplectic JJ-algebra. According to Theorem \ref{the4.1}, it suffices to search for an admissible pair $(\varphi,a)$  of $\K^{2n+1}$ that satisfies  $\al\circ \varphi=0$ and $\om(\varphi(x),a)=0$, for any $x\in\J$. Consider for example
		\[\left\{
	\begin{array}{rlr}
\varphi(e_{2n+1})&=\sum_{i=1}^{2n}\varphi_i e_i\esp \varphi(e_i)=0,& 1 \leqslant i\leqslant 2n,\\
a&=\sum_{i=1}^{2n}a_i e_i,&\\
\varphi_{n+i} a_i&=\varphi_i a_{n+i}, &1 \leqslant i\leqslant n.
	\end{array}
	\right.\]
In this case,  a direct calculation yields to
	\[\left\{
\begin{array}{ll}
\lambda(e_i)=-\frac12 a_{n+i},& 1 \leqslant i\leqslant n\\
\lambda(e_i)=\frac12 a_{i},& n+1 \leqslant i\leqslant 2n\\
\end{array}
\right. \esp \left\{
\begin{array}{ll}
	\theta(e_i,e_j)=\theta(e_{2n+1},e_{2n+1})=0,& 1 \leqslant i,j\leqslant 2n\\
	\theta(e_{2n+1},e_i)=-\varphi_{i+n},& 1 \leqslant i\leqslant n\\
		\theta(e_{2n+1},e_i)=\varphi_{i},& n+1 \leqslant i\leqslant 2n.
\end{array}
\right.\]
Finally, the vector space  $\tilde\J=\langle e_{2n+3} \rangle\oplus\K^{2n+1}\oplus \langle e_{2n+2} \rangle$, endowed with the product
\[\left\{
\begin{array}{ll}
e_{2n+1}\cdot e_i= -\varphi_{n+i}e_{2n+2},& 1 \leqslant i\leqslant n\\
e_{2n+1}\cdot e_i= -\varphi_{i}e_{2n+2},& n+1 \leqslant i\leqslant 2n\\
e_{2n+3}\cdot e_i= -\frac12 a_{n+i}e_{2n+2},& 1 \leqslant i\leqslant n\\
e_{2n+3}\cdot e_i= \frac12 a_{i}e_{2n+2},& n+1 \leqslant i\leqslant 2n\\
e_{2n+3}\cdot e_{2n+3}= a,& 
\end{array}
\right.\]
and the following structure
\[\tilde\al=e^{2n+1}\esp \tilde\om=\om+e^{2n+2}\we e^{2n+3},\]
is a cosymplectic double extension  of the trivial  cosymplectic JJ-algebra $(\K^{2n+1},\om,\al)$.	
\section{Five-dimensional cosymplectic JJ-algebras}
	We use Theorem \ref{theo1} and Proposition~\ref{Pr3} to give a complete classification of five-dimensional cosymplectic JJ-algebras.
	\begin{theo}	
		Any five-dimensional cosymplectic non trivial JJ-algebra is isomorphic to one of the following JJ-algebras:
		\begin{enumerate}				
			\item[$\J_{5,1}:$]  $e_1\cdot e_1 = e_2$,\; $e_1\cdot e_3 = e_4$,\; $e_1\cdot e_5=e_3$,\; $e_2\cdot e_5=-2e_4$, 
				\item[$\J_{5,2}:$] $e_1\cdot e_1 = e_2$,\; $e_1\cdot e_3 = e_4$,\; $e_1\cdot e_5=e_2$,\; $e_3\cdot e_5=2e_4$,
		
			\item[$\J_{5,3}:$] $e_1\cdot e_1 = e_2$,\; $e_1\cdot e_3 = e_4$,			
				\end{enumerate}	
endowed	 with $\al=e^5, \;\om=e^{14}+2e^{23}.$
	\end{theo}
	\begin{proof}
	On the one hand, let $(\h, \om_\h)$ be the unique four-dimensional symplectic non-trivial JJ-algebra defined by $e_1\cdot e_1 = e_2$,\; $e_1\cdot e_3 = e_3\cdot e_1 = e_4$, where $\mathfrak{B}:=\{e_1, e_2, e_3, e_4\}$ is a	basis of the vector space $\h$ and	 $\om_h=te^{14}+2te^{23}$ for $t\in\K-\{0\}$.  Let $\D : \h\lr\h$ be a linear map represented in the base  $\mathfrak{B}$ by the matrix $\D=(a_{i,j})_{1\leqslant i,j \leqslant4}$. A direct calculation shows that  $\D$ is an anti-derivation if and only if $a_{i,j}$ satisfies
		\[
		\left \{
		\begin{array}{l @{=} c}
			a_{1,1}+a_{4,4}+a_{3,3} & 0 \\
			a_{2,2} - 2 a_{4,4}-2 a_{3,3} &0\\
			a_{4,2}+2 a_{3,1}	&0\\
			a_{1,2}=a_{1,3}=a_{1,4}&0\\
			a_{2,4}=a_{3,2}=a_{3,4}&0.
		\end{array}
		\right.
		\]
		In addition, if  $\om_\h(\D(e_i),e_j)=\om_\h(e_i,\D(e_j))$ holds for $1\leqslant i,j \leqslant4$, then we have		
		\[
		\left \{
		\begin{array}{l @{=} c}
			2a_{2,1}-a_{4,3} & 0 \\
			2a_{4,4}+a_{3,3} &0\\
			a_{4,1}=a_{2,3}&0.
		\end{array}
		\right.
		\]
		It follows from the condition $\D^2=0$ that $a_{4,4} = 0$.
		
		Therefore, any anti-derivation $\D$ has the following form
		\[\D=\begin{pmatrix}
			0 & 0 & 0 & 0 \\
			a_{2,1} & 0 & 0 & 0\\
			a_{3,1} & 0 & 0 & 0 \\
			0 & -2 a_{3,1} & 2 a_{2,1} & 0 
		\end{pmatrix},
		\]
		Furthermore, one can distinguish  three cases.
		\begin{enumerate}	
			\item If $a_{3,1}\not=0$, one can take $t=a_{3,1}$ the following automorphism $T$ of $\h$, satisfies
			\[T=\begin{pmatrix}
				1 & 0 & 0 & 0\\
				0 & 1 & -\frac{a_{2,1}}{a_{3,1}} & 0 \\
				0 & 0 & \frac{1}{a_{3,1}} & 0 	\\
				0 & 0 & 0 & \frac{1}{a_{3,1}} 
			\end{pmatrix},\quad T.\D.T^{-1}=\begin{pmatrix}
				0 & 0 & 0 & 0\\
				0 & 0 & 0 & 0 \\
				1 & 0 &0 & 0 	\\
				0 & -2 & 0 &0 
			\end{pmatrix} \esp T_{*}\om_h=e^{14}+2e^{23}.\]
			Which establishes the cosymplectic JJ-algebra $(\J_{5,1},e^{14}+2e^{23},e^5)$.
			\item If $a_{3,1}=0$ and $a_{2,1}\not=0$ , one can take $t=a_{2,1}$ the following automorphism $T$ of $\h$, satisfies
			\[T=\begin{pmatrix}
				\frac{1}{a_{2,1}} & 0 & 0 & 0\\
				0 & \frac{1}{a_{2,1}^{2}} & 0 & 0\\
				0 & 0 & 1 & 0 \\
				0 & 0 & 0 & \frac{1}{a_{2,1}}
			\end{pmatrix},\quad T.\D.T^{-1}=\begin{pmatrix}
				0 & 0 & 0 & 0\\
				1 & 0 & 0 & 0 \\
				0 & 0 &0 & 0 	\\
				0 & 0 & 2 &0 
			\end{pmatrix} \esp T_{*}\om_h=e^{14}+2e^{23}.\]
			Which gives the cosymplectic JJ-algebra $(\J_{5,2},e^{14}+2e^{23},e^5)$.
			\item If $a_{3,1}=0$ and $a_{2,1}=0$. We have directly the cosymplectic JJ-algebra $(\J_{5,3},e^{14}+2e^{23},e^5)$.
		\end{enumerate}	
		
	On the other hand, let us consider $(\K^4,\om,\D)$ to be a four-dimensional symplectic trivial JJ-algebra, where $\D$ is an endomorphism of $\h_0$ such that $\D^2=0$. Then there exists a basis of  $\K^4$ where $\D$ has one of the following canonical forms:
		\[\D_0=(0),\quad \D_1=\begin{pmatrix}
			0 & 1 & 0 & 0\\
			0 & 0 & 0 & 0 \\
			0 & 0 &0 & 0 	\\
			0 & 0 & 0 &0 
		\end{pmatrix} \esp  \D_2=\begin{pmatrix}
			0 & 1 & 0 & 0\\
			0 & 0 & 0 & 0 \\
			0 & 0 &0 & 1 	\\
			0 & 0 & 0 &0
		\end{pmatrix} .\]
		Let $\om=\sum_{1\leqslant i,j \leqslant4} a_{ij}e^{ij}$ be an arbitrary tow form.
		A direct calculation shows that  if \[\om(\D_2(e_i),e_j)=\om(e_i,\D_2(e_j))\] holds for $1\leqslant i,j \leqslant4$. Then, we obtain 
		$a_{12}=a_{13}=a_{34}=0$ and $a_{14}=a_{23}$.	Therefore, $\om=a_{23}e^{14}+a_{23}e^{23}+a_{24}e^{24}$.		
		 In addition, $\om$ is non-degenerate, then  $a_{23}\not=0$.  
		 
		 The following automorphism $T$ of $\K^4$, satisfies
		\[T=\begin{pmatrix}
			-	\frac{1}{a_{2,3}} & 0 & 0 & 0\\
			0 & -	\frac{1}{a_{2,3}} &0 & 0 \\
			0 & 0 & 1 &-\frac{a_{2,4}}{a_{2,3}}	\\
			0 & 0 & 0 &1
		\end{pmatrix},\quad T.\D_2.T^{-1}=\D_2 \esp T_{*}\om=e^{14}+e^{23}.\]
		Which establishes the cosymplectic JJ-algebra $(\J_{5,0},\om_0=e^{14}+e^{23},\al_0=e^5)$.
		
		In the same way, 	$\om(\D_1(e_i),e_j)=\om(e_i,\D_1(e_j))$ holds for $1\leqslant i,j \leqslant4$ we find that $a_{12}=a_{13}=a_{14}=0$. In this case, $\om$ is degenerate. It is clear that for $ \D_0$ we obtain the trivial five-dimensional cosymplectic JJ-algebra.  
		Note that the JJ-algebras  $\J_{5,1}$, $\J_{5,2}$ and $\J_{5,3}$ are non-isomorphic. In fact it suffices  to compare the dimensions of their centers and  derived sub-algebras. Otherwise the  cosymplectic JJ-algebras  $(\J_{5,0},\om_0,\al_0)$ and $(\J_{5,2},\om,\al)$ are isomorphic by the following isomorphism $\Phi : \J_{5,0} \lr \J_{5,2}$, which is given by
		\[\Phi(e_1)=2e_2-\frac12e_5,\;\Phi(e_2)=2e_1,\;\Phi(e_3)=e_4,\;\Phi(e_4)=\frac12e_3,\;\Phi(e_5)=e_5.\] 
	\end{proof}
	
\begin{remark}
We finish by comparing our JJ-algebras  with the ones given in \cite{Y}, by the following table. 
\begin{center}
\begin{tabular}{|c|c|c|}
	\hline
$\J_{5,1}$	&$\J_{5,2}$ &$\J_{5,3}$ \\
	\hline
$\J_{8,5}$	& $\J_{6,5}$& $\J_{1,4}\oplus\K$\\
	\hline
\end{tabular}
\end{center}
\end{remark}

\end{document}